\documentclass{amsart}
\usepackage{amscd,amsmath,amssymb,amsfonts,verbatim, xcolor,graphicx}
\usepackage[all]{xy}
\allowdisplaybreaks
\usepackage{calrsfs}
\DeclareMathAlphabet{\pazocal}{OMS}{zplm}{m}{n}
\makeatletter
\newcommand*\dotp{\mathpalette\dotp@{.5}}
\newcommand*\dotp@[2]{\mathbin{\vcenter{\hbox{\scalebox{#2}{$\m@th#1\bullet$}}}}}
\makeatother

\makeatletter
\newcommand*\bigcdot{\mathpalette\bigcdot@{.5}}
\newcommand*\bigcdot@[2]{\mathbin{\vcenter{\hbox{\scalebox{#2}{$\m@th#1\bullet$}}}}}
\makeatother

\newtheorem{theorem}{Theorem}[subsection]

\newtheorem{lemma}[theorem]{Lemma}  

\theoremstyle{definition}

\newtheorem*{conjecture*}{Conjecture}
\newtheorem*{remark*}{Remark}
\newtheorem{example}[theorem]{Example}

\numberwithin{equation}{subsection}

\newcommand{\isomto}{\overset{\sim}{\rightarrow}}

\title{ I\MakeLowercase{nfinitesimal} D\MakeLowercase{ilogarithm} S\MakeLowercase{atisfies} C\MakeLowercase{luster} I\MakeLowercase{dentities}}

\author{S\MakeLowercase{inan} \"{U}\MakeLowercase{nver}}
\address{Ko\c{c} University, Mathematics Department. Rumelifeneri Yolu, 34450, Istanbul, Turkey}
\email{sunver@ku.edu.tr}

\begin{document}

\maketitle
\noindent

\begin{abstract}
In this paper, we show that the infinitesimal dilogarithm and Kontsevich's $1\frac{1}{2}$-logarithm function satisfies the identities which result from periods in cluster patterns. We also prove that these cluster identities are a consequence of the pentagon relation in the infinitesimal case.
\end{abstract}

\section{Introduction} 

The dilogarithm function, defined by the power series 
$$
{\rm Li}_{2}(z):=\sum _{1\leq n}\frac{z^n}{n^2}\;\;\;\;\;\; {\rm for }\; |z|<1,
$$
has a  history dating back to the 18th century. It has appeared more recently in the context of regulators in algebraic $K$-theory, volumes of hyperbolic manifolds, number theory, and mathematical physics \cite{zag}. It has an analytic continuation to the complex plane as a multi-valued function. It has single-valued versions: the Rogers dilogarithm ${\rm L}(z):={\rm Li}_{2}(z)+\frac{1}{2}\log(z)\log (1-z)$ and the Bloch-Wigner dilogarithm $D(z):=\mathfrak{Im}({\rm Li}_{2}(z))+\arg(1-z)\log|z|.$ These functions, and their higher-weight analogs, satisfy certain functional equations which are important in motivic cohomology; cf. \cite{geo}. The most famous of these equations in the case of the dilogarithm is the pentagon (or five-term) relation \cite{zag}: 
$$
D(x)-D(y)+D(y/x)-D((1-x^{-1})/(1-y^{-1}))+D((1-x)/(1-y))=0.
$$

There is another set of functional equations for the dilogarithm that come from cluster algebras. It is proven that there is a dilogarithm identity corresponding to each period in a cluster pattern by Chapoton \cite{cha}, Nakanishi \cite{nak-id} et al. For a detailed account of the history and references, as well as the proof, we refer the reader to the beautiful survey \cite{nak}.

An infinitesimal version of the dilogarithm was defined and studied in \cite{u1} (cf. \cite[\textsection 3]{inf-dilog}), and was shown to define a regulator from  algebraic $K$-theory. For a ring $A, $ let $A^{\times}$ denote the invertible elements in $A,$  $A^{\flat}:=\{a| a(1-a) \in A^{\times} \}$ and $A_{m}$ denote the truncated polynomial ring $A[t]/(t^m).$  For a field $k,$ of characteristic 0, and  $1<m<w<2m,$ we defined the infinitesimal dilogarithm  $\ell i _{m,w}$ as a function 
$$
\ell i _{m,w}:  k_{m} ^{\flat} \to k.
$$ 
When $k$ is a field of characteristic $p>2,$  we defined a characteristic $p$ version of the infinitesimal dilogarithm in \cite{uao}: 
$$
\ell i_{2} ^{(p)}:k_{2}^{\flat} \to k,
$$
by modifying Kontsevich's $1\frac{1}{2}$-logarithm \cite{kont}. These functions were used to construct infinitesimal invariants of cycles. We refer the reader to the survey \cite{survey} for an overview of these variants of the dilogarithm and the relevant literature.

In this paper, we prove that similar to the case of the ordinary dilogarithm above, the infinitesimal dilogarithms  $\ell i_{m,w}$ and $\ell i_{2} ^{(p)}$ satisfy the relations associated to a period in a cluster pattern. For $\ell i_{m,w}$ this is proved  in Theorem \ref{theorem-0-identity} of \textsection 2.1.1; and for $\ell i_{2} ^{(p)},$ or equivalently the Kontsevich dilogarithm, this is proved in Theorem \ref{theorem-p} of \textsection 2.2.

It is expected that  all the relations among the values of the dilogarithm whose arguments are  rational functions of several variables come from the pentagon relation  \cite{zag}. In the case of relations involving one variable, this was proven by Wojtkowiak \cite[Theorem 4.4]{wojt}. In  general, this is still an open question.  We prove that in the infinitesimal case, the pentagon relation in fact implies the cluster  relations above. In other words, if we have any function $f:k_{m} ^{\flat} \to k,$ which satisfies the pentagon relation, and assumes the value 0 on the constant elements, i.e. on $k^{\flat} \subseteq k_{m} ^{\flat},$ then $f$ satisfies the cluster  relations above.  This result is stated as Theorem \ref{inf-red-theorem} of \textsection 2.1.2. We call this the infinitesimal reduction theorem, in slight analogy with the classical case \cite{nak}. 

The reader might think that since we know the pentagon relation and the infinitesimal reduction theorem, it is unnecessary to prove the infinitesimal cluster relations. In fact the reverse holds true: we are using the infinitesimal cluster relations for the dilogarithm  to prove that they are consequences of the pentagon relation. The additional ingredient in the proof of the infinitesimal reduction theorem is the  main theorem of \cite{u1}, which expresses the homology of the infinitesimal part of the Bloch complex in terms of cyclic homology.

\section{Infinitesimal cluster relations}

\subsection{Inifinitesimal dilogarithm}

We review the theory of infinitesimal dilogarithm. The main references for the definitions and the results in this section are \cite[\textsection 3]{inf-dilog} and \cite{u1}.  In this section, let $k$ be a field of characteristic 0 and $k_{\infty}:=k[[t]]$ denote the ring of formal power series with coefficients in $k.$ For $u \in tk_{\infty},$ we have 
$$
e^{u}=\sum _{0\leq n} \frac{u^n}{n!} \;\;\;\;\;\;{\rm and}\;\;\;\;
\log (1+u)=\sum _{0<n} (-1)^{n+1}\frac{u^n}{n}.
$$  
The latter can be modified to obtain a homomorphism 
$
\log^\circ : k_\infty^\times \to k_\infty
$
defined as 
\[
\log^\circ(\alpha) := \log\left(\frac{\alpha}{\alpha(0)}\right).
\]
Here $\alpha(0) \in k$ denotes the constant coefficient of $\alpha.$ 
If \( q = \sum_{0 \le i} q_i t^i \in k_\infty \) and \( 1 \le a \), then $$
q|_a := \sum_{0 \le i < a} q_i t^i \in k_{\infty}
$$
denotes the truncation of $q$ to its first \( a \) terms, 
\[\frac{\partial q}{\partial t}:=\sum _{0\leq i} iq_{i}t^{i-1}\]
denotes 
the formal partial derivative of $q$ with respect to $t,$  and \( t_a(q) := q_a \) denote the coefficient of $t^a$ in $q.$

For $\alpha \in k_{\infty}^{\flat},$ there exist unique $s \in k^{\flat}$ and $u \in tk_{\infty}$ such that $\alpha =se^{u}.$  We defined the infinitesimal dilogarithm $\ell i _{m,w}:k_{\infty} \to k$ by the formula
\[
\ell i_{m,w}(se^u) := t_{w-1} \left( \log^\circ(1 - se^{u|_m}) \cdot \left. \frac{\partial u}{\partial t} \right|_{w - m} \right),
\]
for \(1< m < w < 2m, \) \cite[\textsection 3]{inf-dilog}, \cite{u1}, and showed that it factors through the canonical projection $k_{\infty} \to k_{m}$ to give a map \[\ell i_{m,w}:k_{m} ^{\flat} \to k,  \] which we denote by the same notation.   

\begin{example} There is only one dilogarithm of modulus $m=2$ and it is given by 
$$
\ell i _{2,3}(s+ut)= -\frac{u^3}{2s^2(s-1)^2},
$$
for $s \in k^{\flat}$ and $u \in k.$ 
There are two dilogarithms of modulus $m=3.$ The one of weight 4 is given by:  
$$
\ell i _{3,4}(s+u_1t+u_2t^2)=\frac{u_1 ^4}{3}\frac{2s-1}{(s-1)^3 s^3}-u_1^2u_2\frac{1}{(s-1)^2s^2}
$$
and the one of weight 5 by: 
$$
\ell i _{3,5}(s+u_1t+u_2t^2)=\frac{u_1^5}{4}\frac{(s-1)^3-s^3}{(s-1)^4s^4}-\frac{u_1^5}{3(s-1)^3s^3} +\frac{5}{3} u_1^3u_2\frac{2s-1}{(s-1)^3s^3}-\frac{5}{2}u_1u_2^2\frac{1}{(s-1)^2s^2},
$$
for $s \in k^{\flat}$ and $u_1, \,u_2 \in k.$
In general, for each $m \geq 2,$ there will be $m-1$ dilogarithms of modulus $m.$

Recall that for a local ring $A$ the Bloch group $B_{2}(A)$ is defined as the quotient of the free abelian group $\mathbb{Z}[A^{\flat}]$ on the symbols $[a],$ for $a \in A^{\flat},$ by subgroup generated by the pentagon relations
\begin{eqnarray}\label{pentagon}
[a]-[b]+[b/a]-[(1-a^{-1})/(1-b^{-1})]+[(1-a)/(1-b)],
\end{eqnarray}
for $a(1-a)b(1-b)(b-a) \in A^{\times}.$ The Bloch complex (of weight two) is defined as:
\[ 
\begin{CD}
B_{2}(A) @>{\delta}>> \Lambda ^{2} _{\mathbb{Z}}A^{\times},
\end{CD}
\]
with $\delta([a]):=(1-a)\wedge a,$ \cite{geo}.

We can rephrase $\ell i_{m,w}$ in terms of the differential in the Bloch complex as follows \cite[Proposition 3.0.1]{inf-dilog}. For $a \geq 1, $ let $\ell_{a}:k_{\infty}^{\times} \to k$ be the homomorphism given by 
\[
\ell _a:= t_{a} \circ \log ^{\circ}.
\]
We proved in \cite{u1}, that the map $g_{m,w}$  from    $B_{2}(k_{\infty})$ to $k$ which sends $[\tilde{\alpha}]$ to 
\begin{eqnarray}\label{dilog-formula}
 g_{m,w} ([\tilde{\alpha}]) =\sum _{1\leq i \leq w-m} i \cdot (\ell _{w-i} \wedge \ell _{i})(\delta(\tilde{\alpha}))   
\end{eqnarray}
has the property that $g_{m,w}([\tilde{\alpha} ])=g_{m,w}([\tilde{\beta} ]),$ if $\tilde{\alpha}|_{m}=\tilde{\beta}|_{m}.$ This implies that, in fact, $g_{m,w}([\tilde{\alpha}])=\ell i_{m,w}(\tilde{\alpha})$ and hence $\ell i _{m,w}$ induces a map 
\begin{eqnarray}\label{dilog-bloch}
\ell i_{m,w}:B_{2}(k_{m})\to k.
    \end{eqnarray}
  
In particular,  $\ell i_{m,w}$ satisfies the pentagon relation (\ref{pentagon}). The sum of the infinitesimal dilogarithms, \(\oplus _{m<w<2m} \ell i_{m,w}\) induce an isomorphism from the infinitesimal part of the $K$-group $K_{3}(k_{m})^{(2)} _{\mathbb{Q}}=(\ker{\delta})_{\mathbb{Q}}$ to $k^{\oplus (m-1)}.$ Here, for an abelian group  $V,$ we let $V_{\mathbb{Q}}:=V\otimes _{\mathbb{Z}}\mathbb{Q}.$ 

There is a natural action of $k^{\times}$ on $k_{m},$ which is obtained by scaling: \( \lambda \times f(t):=f(\lambda t),\) for $\lambda \in k^{\times},$ $f(t) \in k_{m}.$ This induces an action of $k^{\times}$ on $B_{2}(k_{m})$ by functoriality.  The dilogarithm $\ell i_{m,w}$ has $\times$-weight $w:$   
\begin{eqnarray}
\ell i_{m,w}(\lambda \times \alpha)=\lambda ^{w} \ell i_{m,w}(\alpha),    
\end{eqnarray}
for $\lambda \in k^{\times}$ and $\alpha \in B_{2}(k_{m}).$     
\end{example}

\subsubsection{Cluster identities for the infinitesimal dilogarithm.} In this section, we are in the set-up of \cite[\textsection 3]{nak}. We follow the notation there, with few differences. 

For  $1\leq n,$ let $\mathbb{T}_{n},$ the $n$-regular tree graph. Suppose that ${\bf \Sigma}=\{\Sigma_{t}=({\bf x}_t,{\bf y}_{t},B_{t} ) \}_{t \in \mathbb{T}_n}$ is cluster pattern of rank $n$ \cite[Definition 2.9]{nak}. Let ${\bf \Upsilon}=\{\Upsilon_t= ({\bf y}_{t},B_{t} )\}_{t\in \mathbb{T}_n}$ the associated $Y$-pattern of ${\bf \Sigma}.$ We choose an arbitrary initial vertex $t_{0} \in \mathbb{T}_{n}$ and  assume that ${\bf \Upsilon}$ is a free $Y$-pattern with free coefficients ${\bf y}_{t_0}={\bf y}=(y_{1}, \cdots, y_{n})$  at $t_{0}$   \cite[\textsection 2.5.2]{nak}.

For a fixed  $\nu \in S_{n},$ suppose that we have a  sequence 

\begin{eqnarray}\label{periodic-sequence}
\Upsilon[0] \xrightarrow{r_0} \Upsilon[1] \xrightarrow{r_1} \cdots \xrightarrow{r_{P-1}} \Upsilon[P],   
\end{eqnarray}
of mutations such that $\nu \Upsilon[0]=\Upsilon[P]$ \cite[\textsection 3.1]{nak}. Here, $r_{i}\in \{1,\cdots,n \}$  represents mutation in the direction of $r_i.$   Such a sequence of mutations is said to be $\nu$-periodic. We  assume that $\Upsilon[0]$ is the initial seed $\Upsilon_{t_0}.$ For $0\leq j<P$ and $1\leq i \leq n,$ let us denote the $i$-th component of the $y$-variable of $\Upsilon[j]$ by $y_i[j].$ Each $y_{i}[j]$ is a rational function of the initial variables $y_{i},$ $1\leq i\leq n,$ with coefficients in $\mathbb{Z}_{\geq 0}.$    Suppose that the diagonal matrix 
$$
\Theta = {\rm diag} (\theta_{1} ^{-1}, \cdots, \theta_{n} ^{-1} ) 
$$
with $\theta_{i} \in \mathbb{Z}_{>0}$ is a skew-symmetrizer for 
$B_{t_0}.$ Such a matrix exists, by the assumptions on a cluster pattern, but is not unique. This set-up gives a functional equation for the Rogers dilogarithm \cite[Theorem 3.5]{nak}, \cite[Theorem 6.1, Theorem 6.8]{nak-per}.

If $A$ is a ring and $\alpha_i \in A,$ for $1\leq i\leq n,$ we  denote by $\alpha _{i}[j]$ the value of the rational function $y_{i}[j]$ at the point $(\alpha_1,\cdots, \alpha_n),$ when this makes sense, i.e. the denominator of   $y_{i}[j]$ when evaluated at $(\alpha_1, \cdots, \alpha_n)$ is invertible in $A.$ 

\begin{lemma}\label{lemma}
Suppose that we are given a $\nu$-periodic sequence of mutations in a cluster pattern as in (\ref{periodic-sequence}). Let $k$ be a field with   ${\rm char}(k)\neq$2. There is a proper algebraic set $X \subseteq \mathbb{A}^{n}_{k}$ inside the $n$-dimensional affine space $\mathbb{A}^{n}_{k}$ over $k$ such that for ${\alpha}_{1}, \cdots, {\alpha} _{n} \in k_{\infty} $ with $({\alpha}_{1}(0), \cdots, {\alpha} _{n} (0)) \in (\mathbb{A}^n \setminus X)(k),$ we have 
$$
\sum_{0\leq j <P}\theta _{r_{j}}\cdot {\alpha} _{r_j}[j] \wedge (1 +{\alpha} _{r_j}[j])=0 \in \Lambda ^{2} k_{\infty} ^{\times}.   
 $$
 
\end{lemma}

\begin{proof}
We continue with the notation of \cite[\textsection 3]{nak}. 
Let $Q(y_{1},\cdots, y_{n})$ denote the product of all the 
all the $F$-polynomials $F_{i,j}$ and the tropicalizations $[y_{i}[j]]$ of $y_{i}[j],$  for $1 \leq i \leq n$ and $0 \leq j <P.$ The proof of 
\cite[Proposition 3.13]{nak} (cf. \cite[Proposition 6.7]{nak-per}) shows that  
\begin{align}\label{constancy-eq}
 \sum _{0\leq j<P}\theta_{r_j}\cdot y_{r_j}[j]\wedge (1+y_{r_j}[j])=0   
\end{align}
in $\Lambda ^2 k[y_1,\cdots,y_n]_{Q} ^{\times}.$  Here $k[y_1,\cdots,y_n]_{Q}$ denotes the localization of $k[y_1,\cdots,y_n]$ at $Q(y_{1},\cdots,y_n)$ which is obtained by inverting $Q(y_{1}, \cdots, y_{n}).$  Let $X$ be the algebraic set defined by $Q(y_{1},\cdots,y_{n}).$ If  ${\alpha}_{1}, \cdots, {\alpha}_{n} \in  k_{\infty}$ has the property that $({\alpha}_{1}(0), \cdots, {\alpha} _{n} (0)) \in (\mathbb{A}^n \setminus X)(k),$ then we have $Q({\alpha}_{1}(0), \cdots, {\alpha} _{n} (0)) \in k^{\times }.$ This, in turn implies that $Q({\alpha}_{1}, \cdots, {\alpha} _{n} ) \in k_{\infty}^{\times }.$ The $k$-homomorphism from $k[y_{1}, \cdots, y_{n}]$ to $k_{\infty}$ that sends $y_{i}$ to ${\alpha}_{i},$ for $1 \leq i \leq n,$ induces a map $\varphi$ from $k[y_{1},\cdots, y_{n}]_{Q}$ to $k_{\infty}.$
Applying $\varphi$ to (\ref{constancy-eq}) gives the identity in the statement of the lemma. 
 \end{proof}
The following is the precise analog, for the infinitesimal dilogarithm,  of the cluster relations for the Rogers dilogarithm. 

\begin{theorem}\label{theorem-0-identity} Suppose that we are given a $\nu$-periodic sequence of mutations in a cluster pattern as in (\ref{periodic-sequence}). Let  $k$ be a field with ${\rm char}(k)=0,$   and  $\alpha_i \in k_{m} ,$ for $1\leq i \leq n,$ such that   for every $0 \leq j <P,$ the corresponding $\alpha _{r_j}[j]$ has the property that $-\alpha_{r_j}[j]  \in k_{m} ^{\flat}.$  Then  we have  
$$
\sum _{0\leq j <P} \theta_{r_j}\cdot \ell i _{m,w}(-\alpha_{r_j}[j])=0.
 $$

\end{theorem}

\begin{proof}  

Let $R$ be the polynomial ring over  $\mathbb{Q}$ generated by the indeterminates $x _{i,e},$ with $1 \leq i \leq n$ and $0\leq e <m,$ let $F$ be the field of fractions of $R$   and $x_i := \sum _{0 \leq e <m} x _{i,e}t^e \in F_m.$ Applying the mutations above appearing in the $\nu$-periodic sequence, we obtain $x_{i}[j] \in F_{m},$ for $1\leq i\leq n$ and $0\leq j<P.$  In order to ease the notation, we put $y_{i}[j]:=x _{i}[j](0)$ and $y_i:=y_{i}[0].$

Let us also put
$$
f(x_{i,e})_{{1\leq i \leq n} \atop {0\leq e <m}} :=   \sum _{0\leq j <P} \theta_{r_j}\cdot \ell i _{m,w}(-x_{r_j}[j]) \in F.
$$
Notice that $f$ is a rational function in the variables $x_{i,e}$ with coefficients in $\mathbb{Q}.$  From the definition of $\ell i_{m,w},$ we see that $f$  has poles only along some irreducible polynomials in $\mathbb{Q}[y_1,\cdots, y_n] \subseteq R.$ 
If $f$ has a pole along an irreducible polynomial $p(y_1,\cdots,y_n)$ in $\mathbb{Q}[y_1,\cdots, y_n]$ then 
there is a $j$  such that the valuation of $y_{r_j}[j]$ or $1+y_{r_j}[j]$ at $p(y_1,\cdots, y_n)$ is non-zero. Denote by $Y$ the algebraic subset of $\mathbb{A}^{n}_{\mathbb{Q}}$ defined by the product of those irreducible polynomials $p(y_{1},\cdots, y_{n})$  such that there is a $j$  with the property that the valuation of $y_{r_j}[j]$ or $1+y_{r_j}[j]$ at $p(y_1,\cdots, y_n)$ is non-zero. We then  have $Y \subseteq X \subseteq \mathbb{A}^{n} _{\mathbb{Q}},$ with $X$ as in Lemma \ref{lemma}. Note that for $\alpha _{i} \in k_{m},$ with $1 \leq i \leq n,$ the condition that $-\alpha _{r_j}[j] \in k_{m} ^{\flat}$ for every $0\leq j<P$ is equivalent to  $(\alpha_{1}(0), \cdots, \alpha _{n}(0) ) \in (\mathbb{A}^{n}_{\mathbb{Q}} \setminus Y)(k)$ and in this case, 
\begin{align}\label{rational-eqn}
  f(\alpha_{i,e})_{{1\leq i \leq n} \atop {0\leq e <m}} =   \sum _{0\leq j <P} \theta_{r_j}\cdot \ell i _{m,w}(-\alpha_{r_j}[j]) \in k,  
\end{align}
where we put $\alpha_{i}=\sum _{0\leq e<m}\alpha_{i,e}t^e.$

Let $\overline{k}$ denote an algebraic closure of $k.$  For $\beta _{i} \in \overline{k}_m$ such that $(\beta_1(0),\cdots, \beta_{n}(0)) \in (\mathbb{A} ^n \setminus X)(\overline{k}),$ we choose  $\tilde{\beta}_{i} \in \overline{k}_{\infty} $ which reduce to  $\beta _i$ modulo $(t^m).$ By  Lemma \ref{lemma}, we have
\begin{eqnarray}\label{theta-eqn}
  \sum_{0\leq j <P}\theta _{r_{j}}\cdot (1+\tilde{\beta} _{r_j}[j]) \wedge \tilde{\beta} _{r_j}[j]=0.     
\end{eqnarray}

For any  $\beta \in \overline{k}_m ^{\flat}$ and any $\tilde{\beta} \in \overline{k}_{\infty} ^{\flat},$ which  reduces to  $\beta $ modulo $(t^m),$ we have, by (\ref{dilog-formula}),
\begin{align}\label{lifting-indentity-0}
\ell i _{m,w}(-\beta)=\sum_{1 \leq i \leq w - m} i \cdot (\ell_{w - i} \land \ell_i)((1+\tilde{\beta}) \wedge \tilde{\beta}).    
\end{align}
This implies that  
$$
\sum _{0\leq j <P} \theta_{r_j}\ell i _{m,w}(-\beta_{r_j}[j])=\sum _{0\leq j <P} \sum _{1\leq i \leq w-m}\theta_{r_j} i \cdot (\ell_{w - i} \land \ell_i)((1+\tilde{\beta}_{r_j}[j]) \wedge \tilde{\beta}_{r_j}[j]).
$$
The right hand side can be rewritten as 
$$
 \sum _{1\leq i \leq w-m}i \cdot (\ell_{w - i} \land \ell_i)\Big ( \sum _{0\leq j <P}\theta_{r_j}\cdot ((1+\tilde{\beta}_{r_j}[j]) \wedge \tilde{\beta}_{r_j}[j])\Big).
$$
We have shown above that the sum in parentheses is equal to 0 in (\ref{theta-eqn}). By (\ref{rational-eqn}), this implies that $f(\beta_{i,e})_{{1\leq i \leq n} \atop {0\leq e <m}} =0,$ for $\beta_{i} \in \overline{k}_{m}$ such that $(\beta_1(0),\cdots, \beta_{n}(0)) \in (\mathbb{A} ^n \setminus X)(\overline{k}).$ Since $\overline{k}$ is algebraically closed and  $f$ is a rational function that does not have poles along $Y, $ this implies that $f(\beta_{i,e})_{{1\leq i \leq n} \atop {0\leq e <m}} =0,$ for $\beta_{i} \in \overline{k}_{m}$ with $(\beta_1(0),\cdots, \beta_{n}(0)) \in (\mathbb{A} ^n \setminus Y)(\overline{k}).$ Using (\ref{rational-eqn}) shows that 
$$
\sum _{0\leq j <P} \theta_{r_j}\cdot \ell i _{m,w}(-\beta_{r_j}[j])=0,
$$
if $(\beta_{1}(0),\cdots, \beta_{n}(0)) \in (\mathbb{A}^{n}\setminus Y)(\overline{k}).$ Since $k\subseteq \overline{k},$ we have the statement in the theorem. 
\end{proof}

\subsubsection{Inifinitesimal reduction problem }


We will now show that every infinitesimal cluster relation  in fact comes from the pentagon relation. More precisely, we have the following theorem. 

\begin{theorem}\label{inf-red-theorem}
Let $f:k_{m} ^{\flat} \to k$  be a function with the following properties:

(i) $f(s)=0,$ if $s \in k^{\flat} \subseteq k_{m} ^{\flat}.$

(ii) $f$ satisfies the pentagon relation (\ref{pentagon}).

\noindent Then $f$ satisfies the infinitesimal cluster relations corresponding to any $\nu$-periodic sequence of mutations in a cluster pattern as in (\ref{periodic-sequence}). More precisely, for $\alpha_{i} \in k_{m},$ $1\leq i \leq n,$ such that $(\alpha_{1}(0), \cdots, \alpha _{n} (0)) \in (\mathbb{A}^n \setminus X)(k),$ we have 
 \begin{align}\label{general-formula}
 \sum _{0\leq j <P} \theta_{r_j}\cdot f(-\alpha_{r_j}[j])=0.
 \end{align}
\end{theorem}

\begin{proof}

If $f$ is as in the statement of the theorem  then $f$ induces a map from the Bloch group $B_{2}(k_{m})$ to $k,$ 
we denote this homomorphism by $\hat{f}.$  Let $\ker (\delta)^{\circ}$ denote the infinitesimal part of $\ker (\delta).$ In other words, if $q:B_{2}(k_{m}) \to B_{2}(k)$ is the  map induced by the canonical projection $k_{m} \to k,$ then $\ker (\delta)^{\circ}:=\ker (\delta)\cap \ker(q).$  

There are idempotents 
$$
\pi_{w} : \ker (\delta)^{\circ} \to \ker (\delta) ^{\circ} 
$$
for $m<w<2m,$ which induce a decomposition 
$$
\ker (\delta) ^{\circ} = \oplus_{m<w<2m} \pi _{w} (\ker (\delta)^{\circ}), 
$$  
cf. \cite[\textsection 4.1.1]{u1}. 
Here $\pi _{w} (\ker (\delta)^{\circ})=\{u \in  \ker (\delta)^{\circ}| \lambda \times u = \lambda ^w u, {\rm \; for \; } \lambda \in \mathbb{Z} \}$   is the weight $w$ part of the $\times$-action, cf. \textsection 2.1 and \cite[Notation 2.0.1]{u1}. The  restriction of $\ell i _{m,w}$  to $\pi _{w} (\ker (\delta)^{\circ}) $ induces a (group) isomorphism $$ 
\gamma_{w}:=\ell i _{m,w}|_{\pi _{w} (\ker (\delta)^{\circ})}:\pi _{w} (\ker (\delta)^{\circ}) \isomto k,
$$
by \cite[Theorem 1.3.1]{u1} and \cite[Theorem 1.3.2]{u1}. 
Let $g_{w}:= \hat{f}|_{\pi _{w} (\ker (\delta)^{\circ})}\circ\gamma_{w}^{-1},$ and 
\[\hat{g}:=\sum_{m <w <2m} g_{w} \circ \ell i_{m,w}:B_{2}(k_{m}) \to k. \] 

We proved in Theorem \ref{theorem-0-identity} that $\ell i _{m,w},$ for $m<w<2m$ satisfy the  identity (\ref{general-formula}). Since $g_{w}$ are group homomorphisms and the identity (\ref{general-formula}) is linear,  we see that $\hat{g}$ also satisfies this identity. Note that  $\hat{f}$ and $\hat{g}$ agree on $(\ker{\delta})^{\circ}$ by construction.  Furthermore, $\hat{f}$ and $\hat{g}$ are both zero on $B_{2}(k)\subseteq B_{2}(k_m), $ the first one by the assumption (i) above, the second one because of the fact that $\ell i _{m,w}$ vanish on the same subgroup. This implies that  $\hat{f}$ and $\hat{g}$ agree on $\ker{\delta}$ and hence $\hat{f}-\hat{g}$ factors through the projection $\delta$ from $B_{2}(k_{m})$ to $\delta(B_{2}(k_{m})) \subseteq \Lambda ^{2} k_{m} ^{\times}.$ Given $\alpha _{i}$ as in the statement of the theorem, by Lemma \ref{lemma}, we have 
$$
\delta(\sum_{0\leq j<P}\theta_{r_j}[-\alpha_{r_j}(j)])=0
$$ 
in $(\Lambda ^{2}k_{m} ^{\times})_{\mathbb{Q}}.$
Therefore, $\hat{f}-\hat{g}$ also satisfies the identity (\ref{general-formula}). This, in turn, implies that $\hat{f}=\hat{g}+(\hat{f}-\hat{g})$ satisfies the identity  (\ref{general-formula}) as well. 
 
\end{proof}

\subsection{Cluster relations for Kontsevich $1\frac{1}{2}$-logarithm } Let $p$ be an odd prime and $R$ a ring of characteristic $p.$ For $s \in R,$ let 
\[
\pounds_1(s) = \sum_{1 \leq i <p } \frac{s^i}{i},
\]
denote Kontsevich's $1\frac{1}{2}$-logarithm as defined in \cite[Definition 4.1]{gang} and \cite{kont}. For $y=s+\alpha t \in R_{2},$ we put $\underline{y}:=s$ and $\overline{y}:=\frac{\alpha}{s(1-s)}.$ Then using the notation of \cite[\textsection 3]{uao}, we have 
$\ell i _2^{(p)}(y)=\overline{y}^p \pounds_1(\underline{y}).$ In \cite{uao}, we showed   that $\ell i _2^{(p)}$ is the  component of  a regulator from $K_{3}(R_2),$ when $R$ is a local ring. Analogously to the maps $\ell _i$  in \textsection 2.1 for characteristic 0, we have the  maps $\ell _i: R_{\infty} ^{\times } \to R,$ for $i<p,$ in characteristic $p.$  This is because the first $p$ terms of the power series expansion of $\log(1+x)$ does not involve $p$ in the denominator. Using these maps,  $\ell i_{2}^{(p)}$ can  be expressed in terms of the differential in the Bloch complex \cite[\textsection 3]{chow-kont}: 
\begin{align}\label{pound-delta}
\ell i_{2} ^{(p)}=(\frac{1}{2} \sum _{1 \leq i <p }i\cdot \ell _{p-i} \wedge \ell _{i}) \circ \delta. 
\end{align}

\begin{theorem}\label{theorem-p} Suppose that we are given a $\nu$-periodic sequence of mutations in a cluster pattern as in (\ref{periodic-sequence}). Let $k$ be a field of characteristic $p>2.$ For $\alpha_{i} \in k_{2} ,$ $1\leq i \leq n,$ with the property that $-\alpha _{r_{j}}[j] \in k_{2}^{\flat},$ for all $0\leq j <P,$  we have 
$$
\sum_{0\leq j<P}\theta_{r_j}\cdot \ell i^{(p)} _{2}(-\alpha _{r_j}[j])=0.
$$ If we 
 put $\beta_{j}:=-\alpha_{r_j}[j],$ this can be rephrased as 
$$
\sum_{0\leq j<P}\theta_{r_j} \cdot\overline{\beta}_j^p \pounds_{1}(\underline{\beta}_j
)=0.
$$
\end{theorem}

\begin{proof}
The proof of this theorem is entirely analogous to that of Theorem \ref{theorem-0-identity}.  Here we use the identity (\ref{pound-delta}) which expresses $\ell i_{2} ^{(p)}$ in terms of a lifting,  instead of the use of (\ref{lifting-indentity-0}) in the proof of Theorem \ref{theorem-0-identity}. The details are omitted.  
\end{proof}

\begin{example} 
(a) The fact that  mutations are involutive gives us the most basic   relation $\ell i _2^{(p)}(y_1^{-1})+\ell i _2^{(p)}(y_1)=0,$ for $y_{1} \in k_{2} ^{\flat}.$   

The periodic set of mutations for an  $A_2$-type cluster algebra given in \cite[Example 3.8]{nak} gives us the functional equation: 
\[
\begin{aligned}
\ell i_{2} ^{(p)}(y_1)+ \ell i_{2} ^{(p)}(y_2(1-y_1))+\ell i_{2} ^{(p)}(y_1^{-1}(1-y_2+y_1y_2))+\ell i_{2} ^{(p)}(y_{1} ^{-1}(1-y_2^{-1})) +\ell i_{2} ^{(p)}(y_2^{-1})=0,
\end{aligned}
\]
by replacing $y_i$ with $-y_i$ in the set of equations in \cite[\textsection 3.3 (3.18)]{nak}.  If we further put $y_{1}=1-x$ and $y_{2}=y/x$ and use the elementary relation  that $\ell i _2^{(p)}(1-z)+\ell i _2^{(p)}(z)=0,$ the above relation can be rewritten as 
\[
\begin{aligned}
\ell i_{2} ^{(p)}(x)- \ell i_{2} ^{(p)}(y)+\ell i_{2} ^{(p)}(\frac{y}{x})-\ell i_{2} ^{(p)}(\frac{1-x^{-1}}{1-y^{-1}}) +\ell i_{2} ^{(p)}(\frac{1-x}{1-y})=0.
\end{aligned}
\]
Putting $x:=r+r(1-r)t$ and $y:=s+s(1-s)t$ gives the famous 4-term functional equation \cite{kont} of the $1\frac{1}{2}$-logarithm:

\[
\begin{aligned}
\pounds_1(r)-\pounds_1(s)+r^p\pounds_1(\frac{s}{r})+(s-1)^p\pounds_1(\frac{1-r}{1-s})=0.
\end{aligned}
\]

(b) Corresponding to the periodic set of mutations for the $B_{2}$-type cluster algebra given in \cite[Example 3.9]{nak}, we obtain the following relation: 

\[
\begin{aligned}
\pounds_1(r_1)+ 2\frac{1-r_1^p-r_2^p}{1-r_2^p+r_1^pr_2^p} \pounds_1 (r_2(1-r_1))+2r_2^p \pounds _{1}(r_{2}^{-1})+ \frac{r_1 ^p-3}{1-r_1^{-p}(1-r_2^{-p})^2}\pounds_1 (r_1^{-1} (1 - r_2 ^{-1})^2)\\
+2\Big(\frac{(1-r_2)(r_2+(r_2^{-1}-1)(2r_1^{-1}+r_{1}^{-1}r_{2}-1 ) )}{(r_1 ^{-1}(1-r_2)(1-r_2^{-1})+r_2)(1-r_1 ^{-1}(1-r_2)(1-r_2^{-1})-r_2)}\Big)^p\pounds_1 (r_1 ^{-1}(1-r_2)(1-r_2^{-1})+r_2)\\
+\Big(\frac{2r_1^{-1}r_{2}(r_2-1+r_{1}(2-r_1-r_2))+(1-r_1^{-1})(1-r_2+r_1r_2)^2}{r_1^{-1}(1-r_2+r_1r_2)^2(1-r_1^{-1}(1-r_2+r_1r_2)^2)} \Big)^p \pounds_1 (r_1^{-1}(1 - r_2 + r_1 r_2)^2)=0.\\
\end{aligned}
\]
\end{example}

\end{document}